\documentclass[10pt]{amsart}

\usepackage{amsmath,amssymb,amsfonts,epsfig,mathrsfs,color}

\newtheorem{thm}{ Theorem}[section]

\newtheorem{lemG}{ Lemma}

\theoremstyle{definition}
\newtheorem{ex}[thm]{ Example}
\newtheorem{rem}[thm]{ Remark}
\newtheorem{defn}[thm]{ Definition}

\newcommand{\R}{\mathbb{R}} 
\newcommand{\rb}{\raisebox}
\newcommand{\ig}{\includegraphics}

\def\sminus{\smallsetminus}
\def\lw{\langle x|D|S \rangle}
\def\lwg{\langle \a|G|S \rangle}
\def\lwa{\langle \a|A|S \rangle}
\def\plw{\langle D|S \rangle}
\def\plwg{\langle G|S \rangle}
\def\plwa{\langle A|S \rangle}
\def\lvw#1#2{\langle #1|#2|S \rangle}
\def\scp#1#2{\langle #1,#2 \rangle}
\def\e{\varepsilon}
\def\a{\alpha}

\def\A{\mathscr{A}}

\def\I{\mathcal{I}}
\def\bp{\risS{-2}{bp}{}{10}{0}{0}}

\newcommand{\Z}{\mathbb{Z}} 
\newcommand\risS[6]{\rb{#1pt}[#5pt][#6pt]{\begin{picture}(#4,15)(0,0)
  \put(0,0){\ig[width=#4pt]{#2.eps}} #3
     \end{picture}}}

\def\ard#1{\risS{-12}{#1}{}{25}{15}{17}}

\begin{document}

\title[Elementary combinatorics of the HOMFLYPT polynomial]{Elementary
   combinatorics of the HOMFLYPT polynomial}
\author{SERGEI CHMUTOV and MICHAEL POLYAK}

\address{The Ohio State University, Mansfield,
1680 University Drive, Mansfield, OH 44906. {\tt
chmutov@math.ohio-state.edu}\linebreak Department of mathematics,
Technion, Haifa 32000, Israel. {\tt polyak@math.technion.ac.il}}

\begin{abstract}
We explore Jaeger's state model for the HOMFLYPT polynomial. We
reformulate this model in the language of Gauss diagrams and use it
to obtain Gauss diagram formulas for a two-parameter family of
Vassiliev invariants coming from the HOMFLYPT polynomial. These
formulas are new already for invariants of degree 3.
\end{abstract}



\maketitle

\newcommand{\unkn}{\rb{-4.2mm}{\ig[width=10mm]{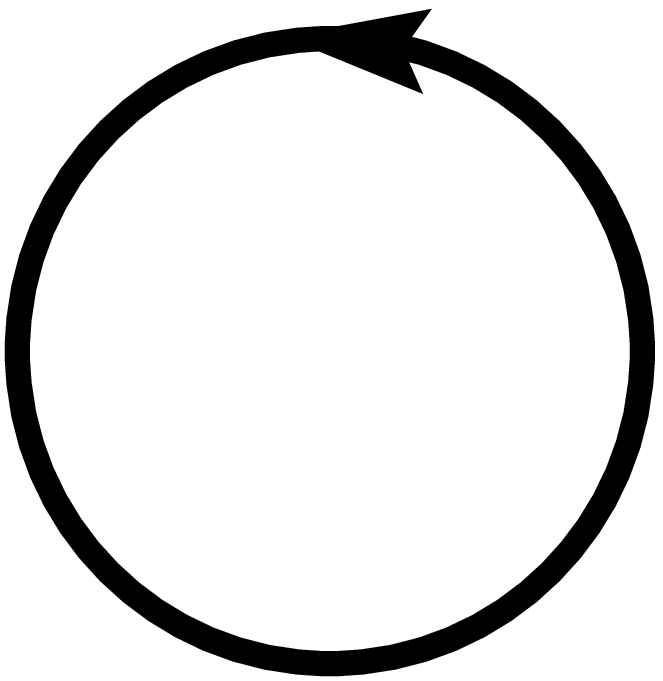}}}
\newcommand{\twoup}{\rb{-4.2mm}{\ig[width=10mm]{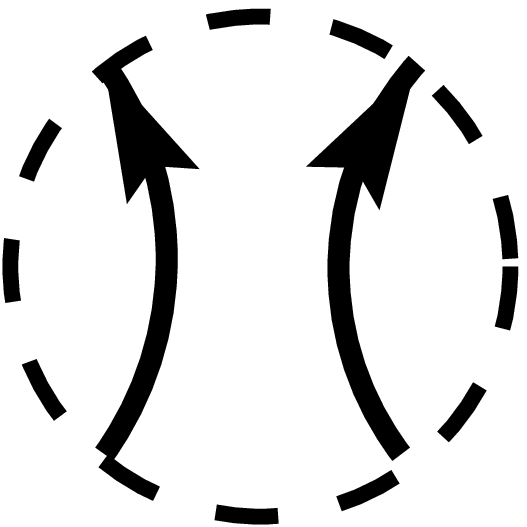}}}
\newcommand{\rlints}{\rb{-4.2mm}{\ig[width=10mm]{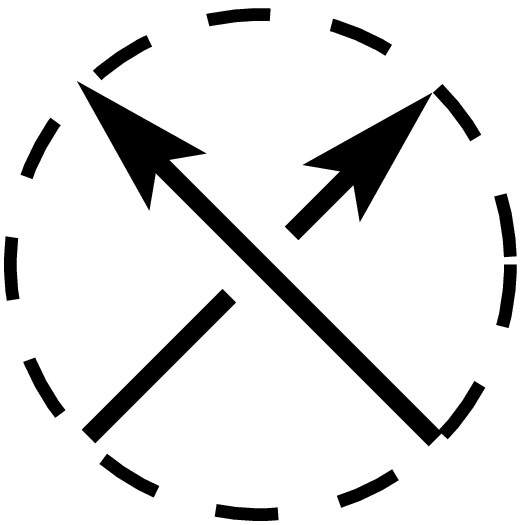}}}
\newcommand{\lrints}{\rb{-4.2mm}{\ig[width=10mm]{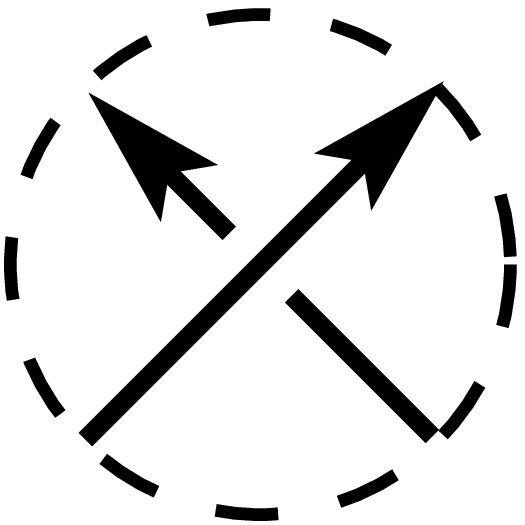}}}

\section*{Introduction} \label{s:intro}
The {\em HOMFLYPT polynomial} $P(L)$ is an invariant of oriented
link $L$. It is defined as the Laurent polynomial in two variables
$a$ and $z$ with integer coefficients satisfying the following skein
relation and the initial condition:
\begin{equation}\label{eq:skein}
aP(\lrints) - a^{-1}P(\rlints) =  zP(\twoup)\ ;\qquad P(\unkn)\quad
= \quad 1\,.
\end{equation}
If $L$ is an unlink with $m$ components then
$P(L)=\Bigl(\frac{a-a^{-1}}{z}\Bigr)^{m-1}$. The proof of the
existence of such an invariant is long and cumbersome. It was
established simultaneously and independently by five groups of
authors \cite{HOM,PT}.

This paper is devoted to Gauss diagram formulas for Vassiliev
invariants coming from the HOMFLYPT polynomial. It is known
\cite{GPV} that any Vassiliev knot invariant may be presented by a
Gauss diagram formula. This type of formulas is the simplest for
computation purposes; however, the algorithm for producing them is
complicated and until recently only few lower degree cases were
described explicitly. The first description of such formulas for an
infinite family of Vassiliev invariants was given in \cite{CKR},
where the coefficients of the Conway polynomial were considered.
This paper generalizes the result of \cite{CKR} to the HOMFLYPT
polynomial.

We use a non-standard change of variables (used formely in
\cite{G2}), leaving $z$ alone and plugging in $a=e^h$ to obtain a
power series $\sum_{k,l}p_{k,l}h^kz^l$. The coefficients $p_{k,l}$
are Vassiliev invariants of degree $\leqslant k+l$, see \cite{G2}.
We give the Gauss diagram formulas for $p_{k,l}$ for arbitrary
$k,l$. These formulas are new already for invariants of degree 3.

The paper is organized in the following way. In Section
\ref{s:homfly} we start from the scheme of \cite{H, LM, PT},
extracting from it an explicit state model for the HOMFLYPT
following \cite{Ja} in Section \ref{s:jaeger}. We then briefly
review the notions of Gauss diagrams in Section \ref{s:gaus-diagr}
and reformulate the state model in these terms in Section
\ref{s:jaeger-ga}. The expansion of $P(L)$ into power series in $h$
and $z$ is considered in Section \ref{s:vas-from-homfly}. In the
same section we  remind the definition of the Gauss diagram formulas
for Vassiliev invariants. Finally, we describe the Gauss diagram
formulas for $p_{k,l}$ in Section \ref{s:result}. In the last
Section \ref{s:example} we analyze low degree cases in details.

Note that using instead of \eqref{eq:skein} the skein relation for
the two-variable Kauffman polynomial, one gets a similar state
model. We plan to consider the resulting Gauss diagram formulas in a
forthcoming paper.

We are grateful to O.~Viro, L.~Traldi, and to the anonymous referee
for numerous corrections to the first version of the paper and
useful remarks. This work has been done when both authors were
visiting the Max-Plank-Institut f\"ur Mathematik in Bonn, which we
would like to thank for excellent work conditions and hospitality.
The second author was supported by a grant 3-3577 of the Israel
Ministry of Science and ISF grant 1261/05.

\section{HOMFLYPT and descending diagrams} \label{s:homfly}

The skein relation \eqref{eq:skein} allows one to calculate the
HOMFLYPT polynomial of a link. Following \cite{H, LM, PT}, this can
be done by ordering a link diagram and then transforming it into a
descending diagram. We call a diagram $D$ {\it ordered}, if its
components $D_1$, $D_2$,\dots ,$D_m$ are ordered and on every
component a (generic) base point is chosen. An ordered diagram is
{\it descending}, if $D_i$ is above $D_j$ for all $i<j$ and if for
every $i$ as we go along $D_i$ starting from its base point along
the orientation we pass each self-crossing first on the overpass and
then on the underpass.

An elementary step of the algorithm computing $P(L)$ consists of the
following procedure. Suppose that $D$ is an ordered diagram and that
the subdiagram $D_1,\dots,D_{i-1}$ is already descending. We go
along $D_i$ (starting from the base point) looking for the first
crossing which fails to be descending. At such a crossing $x$ we
change it using the skein relation. Namely, depending on the sign
$\e$  (the local writhe) of the crossing, we express $P(D)$ as
\begin{equation}\label{eq:descend}\begin{array}{ccl}
P(\lrints)&=&a^{-2}P(\rlints) + a^{-1}zP(\twoup)\vspace{5pt}\\
P(\rlints)&=&a^{2}P(\lrints) - azP(\twoup)
\end{array}\end{equation}
Denote the corresponding diagrams $D^\e$, $D^{-\e}$, $D^0$.

The ordering of $D=D^\e$ induces an ordering of $D^{-\e}$ (in an
obvious way); the ordering of $D^0$ requires some explanation. If
$x$ was a crossing of $D_i$ with $D_j$, $j>i$, then these two
components merge into a single component $D^0_i$ of $D^0$, with a
base point being the base point of $D_i$. If $x$ was a self-crossing
of $D_i$, then $D_i$ splits into two components: $D^0_i$, which
contains the base point of $D_i$, and $D^0_{i+1}$, where we choose
the base point in a neighborhood of $x$. In both cases the order of
remaining components shifts accordingly.

The diagrams $D^{-\e}$, $D^0$ are ``more'' descending than $D^\e$.
At the next step we apply the same procedure to each of them.

\begin{ex}\label{ex:1a} For the trefoil $3_1$ the algorithm consists
of two steps, illustrated in the figure below. The diagram $D^+$
appearing in the first step is already descending; the diagram $D^0$
is not, so the second step is needed to transform it.

$$\risS{-15}{31alg}{\put(20,60){\mbox{$D^-$}}
         \put(-60,85){\mbox{Step 1:}}\put(-60,20){\mbox{Step 2:}}
         \put(85,60){\mbox{$D^+$}}\put(150,60){\mbox{$D^0$}}
         \put(25,-5){\mbox{$D^-$}}\put(98,-5){\mbox{$D^+$}}
         \put(140,-5){\mbox{$D^0$}}}{160}{90}{25}
$$
Hence $P(3_1)=a^2\cdot 1 - az\Bigl(
a^2\cdot\frac{a-a^{-1}}{z}-az\cdot 1\Bigr) = (2a^2-a^4)+a^2z^2$.
\end{ex}

\section{State model reformulation} \label{s:jaeger}

The state model of \cite{Ja} for the HOFMLYPT polynomial is a
convenient reformulation of the algorithm of Section \ref{s:homfly}.

A {\it state} $S$ on a link diagram $D$ is a subset of its
crossings. The HOMFLYPT polynomial is going to be a sum over the
states. Let $D(S)$ be the link diagram obtained by smoothing every
crossing in $S$ according to orientation and $c(S)$ be the number of
its components. We will not use the topology of $D(S)$, however its
combinatorics will determine the contribution of the state $S$ to
the state sum. The contribution will be a product of a global weight
of the state as a whole, $\bigl(\frac{a-a^{-1}}{z}\bigr)^{c(S)-1}$
and local weights of crossings of the diagram.

The ordering of $D$ induces an ordering of $D(S)$ (in the way
explained in Section \ref{s:homfly} above) and thus determines a
tracing of the link $D(S)$. The local weight $\lw$ of a crossing $x$
of $D$ depends on the first passage of a neighborhood of $x$ in the
tracing and on the sign $\e$ of $x$. Namely, if $x$ is in $S$ and we
approach $x$ first time on an overpass of $D$ then $\lw=0$ (since
such a situation does not occur in the above algorithm). If we
approach $x$ on an underpass of $D$ then $\lw=\e a^{-\e}z$ (i.e.,
the coefficient of $D^0$ in \eqref{eq:descend}). In the case if $x$
does not belong to $S$ and we approach $x$ first time on an overpass
then $\lw=1$ (since in the above algorithm we do not apply the skein
relation to $x$). If we approach $x$ on an underpass then
$\lw=a^{-2\e}$ (i.e., the coefficient of $D^{-\e}$ in
\eqref{eq:descend}). These assignments can be summarized in the
following figure.
$$\begin{array}{c||c|c|c|c}
\rb{5pt}{First
passage:}&\risS{10}{br2}{}{25}{0}{0}&\risS{0}{br1}{}{25}{20}{3}&
 \risS{0}{br4}{}{25}{0}{0}&\risS{0}{br3}{}{25}{0}{0} \\ \hline\hline
\risS{-8}{cr_p}{}{25}{15}{10} & 0 & a^{-1}z & 1 &  a^{-2}\\ \hline
\risS{-8}{cr_m}{}{25}{15}{10} & -az & 0 & a^2 & 1
\end{array}
$$
Denote by $\plw:=\prod_x \lw$ the product of local weights of all
crossings. For a link $L$ with a diagram $D$ we have
\cite[Proposition 2]{Ja}:
$$P(L)=\sum_S\ \ \plw\cdot \left(\frac{a-a^{-1}}{z}\right)^{c(S)-1}
$$

\begin{ex} Consider a based trefoil diagram $D$ and a state $S$
consisting of one crossing $\{x_1\}$.
$$D=\risS{-15}{31}{
         \put(17,-3){\mbox{\scriptsize $x_1$}}
         \put(34,23){\mbox{\scriptsize $x_2$}}
         \put(0,26){\mbox{\scriptsize $x_3$}}}{40}{20}{20} \hspace{2cm}
  D(S)=\risS{-20}{conwlin}{}{40}{20}{0}
$$
The tracing of $D(S)$ first approaches the crossing $x_1$ on the
strand which was an underpass in $D$. So its weight will be
$\lvw{x_1}{D}=-az$. Similarly the weights of the other two crossings
are $\lvw{x_2}{D}=a^2$ and $\lvw{x_3}{D}=1$. So the total
contribution from this state will be equal to
$-a^3z\bigl(\frac{a-a^{-1}}{z}\bigr) = a^2-a^4$. The next table
shows the contributions from all eight states. Non-zero weights come
from states corresponding to descending diagrams appearing in the
end of the algorithm, see Example \ref{ex:1a}.
$$\begin{array}{c|c|c|c|c|c|c|c}
\emptyset & \{x_1\} & \{x_2\} & \{x_3\} & \{x_1,x_2\} & \{x_1,x_3\}
&
   \{x_2,x_3\} & \{x_1,x_2,x_3\} \\ \hline \makebox(0,12){}
a^2& a^2-a^4 & 0 & 0 & a^2z^2 & 0 & 0 & 0
\end{array}
$$
So we recover the result of Example \ref{ex:1a}: $P(3_1)=(2a^2-a^4)+
z^2a^2$.
\end{ex}

\medskip
\begin{rem}
Smoothing a crossing from a state $S$ changes the number of
components by one. Hence the cardinality $|S|$ and the difference
$m-c(S)$ (where $m$ is the number of components of $D$) are
congruent modulo 2. Therefore the HOMFLYPT polynomial $P(L)$ is even
in each of the variables $a$ and $z$ if $m$ is odd, and it is an odd
polynomial if $m$ is even.
\end{rem}
\begin{rem}
The negative powers of $z$ come from the factors
$\bigl(\frac{a-a^{-1}}{z}\bigr)^{c(S)-1}$. A smoothing of a crossing
$x\in S$ may increase $c(S)$ by one, however this increment will be
compensated by a local weight $\lw$. As a consequence we have that
the lowest power of $z$ in the HOMFLYPT polynomial of $L$ is at
least $-m+1$. In particular the HOMFLYPT polynomial $P(K)$ of a knot
$K$ is a genuine polynomial in $z$, i.e., does not contain terms
with negative powers of $z$.
\end{rem}

\section{Gauss diagrams} \label{s:gaus-diagr}

\begin{defn}
Gauss diagrams provide an alternative and more combinatorial way to
present links. For a link diagram $D$ consider a collection of
(counterclockwise) oriented circles parameterizing it. Two preimages
of a crossing of $D$ we unite in a pair and connect them by an arrow
pointing from the overpassing preimage to the underpassing one. To
each arrow we assign a sign $\pm1$ of the corresponding crossing.
The result is called the {\it Gauss diagram} $G_D$ of the link
diagram $D$. A link can be uniquely reconstructed from the
corresponding Gauss diagram \cite{GPV}.
\end{defn}

For example, a Gauss diagram of the trefoil looks as follows.
$$D=\ \risS{-15}{31}{}{40}{20}{20} \hspace{3cm}
G_{D}=\ \risS{-15}{gd-31bp}{
         \put(-135,-15){\mbox{\tt A knot and its Gauss diagram}}}{40}{27}{35}\label{d3-1}
$$

Not every diagram with arrows is realizable as a Gauss diagram of a
classical link. For example, $\risS{-5}{gauss-nr}{}{15}{14}{8}\ $ is
not realizable regardless of signs of its arrows. An {\it abstract
Gauss diagram}, or an {\it arrow diagram} is a generalization of a
notion of Gauss diagram, in which we forget about realizability. In
other words, an arrow diagram consists of a number of oriented
circles with several arrows connecting pairs of distinct points on
them. The arrows are equipped with signs $\pm1$. We consider these
diagrams up to orientation preserving diffeomorphisms of the
circles.

We are going to work with {\it ordered Gauss diagrams}, i.e. Gauss
diagrams with ordered circles and a base point $\bp_1, \bp_2, \dots,
\bp_m$\, on each circle corresponding to an ordering of $D$.
Similarly, an {\it ordered arrow diagram} is an arrow diagram
equipped with an ordering of the circles and a base point (different
from the end points of the arrows) on each of them.

Two Gauss diagrams represent isotopic links if and only if they are
related by a finite number of Reidemeister moves (see, for example,
\cite{GPV,Oll,CDbook}).
$$\Omega_1:\ \risS{-15}{virrI}{
            \put(-2,5){\mbox{$\scriptstyle \e$}}
            \put(85,5){\mbox{$\scriptstyle \e$}}
             }{100}{18}{15} \hspace{2cm}
\Omega_2:\ \risS{-15}{virrII1}{
            \put(6,18){\mbox{$\scriptstyle \e$}}
            \put(22,18){\mbox{$\scriptstyle -\e$}}
             }{95}{0}{0}
$$
$$\Omega_3: \risS{-18}{virrIII}{}{120}{20}{22}\ .
$$
Note that the segments involved in $\Omega_2$ or $\Omega_3$ may lie
on the different components of the link. So the order in which they
are traced along the link may be arbitrary.

An {\it ordered link} is an equivalence class of ordered Gauss
diagrams modulo Reidemeister moves which do not involve base points.
For $m=1$ this notion is equivalent to the notion of long knots
defined as embeddings of $\R$ into $\R^3$ which coincide with a
standard embedding (say, an $x$-axis) outside a compact. It is well
known that for classical knots the theories of long and closed knots
coincide.

\section{State models on Gauss diagrams} \label{s:jaeger-ga}

All notions and constructions of Section \ref{s:jaeger} have a
straightforward translation to the language of Gauss diagrams.

A {\it state} $S$ on an abstract Gauss diagram $G$ is a subset of
its arrows. Let $G(S)$ be the abstract Gauss diagram obtained by
doubling every arrow in $S$ as in the figure
$$\risS{-6}{arrow}{}{30}{0}{5}\qquad
 \risS{-2}{totor}{}{25}{0}{0}\qquad\risS{-6}{darrow}{}{30}{0}{0}\ ,
$$
and let $c(S)$ be the number of its circles. The ordering of $G$
induces an ordering of $G(S)$. The local weight $\lwg$ of an arrow
$\a$ of $G$ in general depends on whether $\a$ belongs to $S$, on
the first passage in a neighborhood of $\a$ in the tracing of
$G(S)$, and on the sign $\e$ of $\a$. Given a table of such local
weights, we denote by $\plwg:=\prod_\a \lwg$ the product of local
weights of all arrows and define a polynomial $P(G)$ by
\begin{equation}\label{eq:PG}
P(G):=\sum_S\ \ \plwg\cdot \left(\frac{a-a^{-1}}{z}\right)^{c(S)-1}
\end{equation}

The table of local weights for the HOMFLYPT state model (readily
taken from Section \ref{s:jaeger}) is shown below.
\begin{equation}\label{eq:lwg}
\begin{array}{c||c|c|c|c}
\rb{7pt}{First passage:}&
 \risS{0}{fp-b}{}{35}{25}{3}&\risS{0}{fp-t}{}{35}{0}{0}&
 \risS{0}{fp-l}{}{35}{0}{0}&\risS{0}{fp-r}{}{35}{0}{0} \\ \hline\hline
\risS{-8}{cr_p-gd}{}{40}{22}{15} & a^{-1}z & 0 & a^{-2} & 1 \\
\hline \risS{-8}{cr_m-gd}{}{40}{22}{15} & -az & 0 & a^2 & 1
\end{array}
\end{equation}

\medskip
\begin{ex} For the Gauss diagram of the trefoil the states with non-zero
weights are the following.
$$\begin{array}{r||c|c|c}
\mbox{States of\quad } \risS{-10}{gd-31bp}{}{30}{20}{18}\ :&
   \risS{-10}{s31-0}{}{30}{0}{0} & \risS{-10}{s31-1}{}{33}{0}{0} &
   \risS{-10}{s31-12}{}{32}{0}{0} \\
\mbox{Weights}\ : & 1\cdot a^2\cdot 1&
    1\cdot (-az)\cdot a^2\cdot \Bigl(\frac{a-a^{-1}}{z}\Bigr)&
    1\cdot (-az)\cdot (-az)
\end{array}
$$
Hence, $P(G)=(2a^2-a^4)+ z^2a^2$.
\end{ex}
The HOMFLYPT polynomial defined by this state model may be called
the {\em descending} HOMFLYPT polynomial. An {\em ascending}
HOMFLYPT polynomial may be defined in a similar way, interchanging
the values of the first two columns and the last two columns in the
table (\ref{eq:lwg}) of local weights. For classical links these two
polynomials coincide.

\section{Vassiliev invariants coming from the HOMFLYPT polynomial} \label{s:vas-from-homfly}

\subsection{HOMFLYPT power series}
A standard way \cite{BN,BL} to relate Vassiliev invariants to the
HOMFLYPT polynomial is to make a substitution $a=e^{Nh}$,
$z=e^h-e^{-h}$ and then take the Taylor expansion of $P(L)$ in the
variable $h$. The coefficient at $h^n$ turns out to be a Vassiliev
invariant of order $\leqslant n$ which depends on a parameter $N$.

In this paper we are working in a different way, following
\cite{G2}. Namely, we substitute $a=e^h$ and take the Taylor
expansion in $h$. The result will be a Laurent polynomial in $z$ and
a power series in $h$. Let $p_{k,l}(L)$ be its coefficient at
$h^kz^l$. It is not difficult to see that for any link $L$ the total
degree $k+l$ is not negative. (It also follows from the Jaeger model
in Section \ref{s:jaeger}.)

\begin{lemG}[\cite{G2}]
$p_{k,l}(L)$ is a Vassiliev invariant of order $\leqslant k+l$.
\end{lemG}

\begin{proof}
Indeed, plugging $a=e^h$ into the skein relation we get
$$P(\lrints)\ -\ P(\rlints)\ = \  zP(\twoup) + h(\mbox{some terms})\ .$$
Since all terms of the HOMFLYPT polynomial have non-negative total
degree in $z$ and $h$, the terms of the right hand side has degree
at least $1$. Therefore, if we change $n+1$ crossings in different
places then the alternating sum of the $2^n$ polynomials will have
the degree of its monomials $\geqslant n+1$. Hence the coefficient
at any degree $n$ term will be zero.
\end{proof}

\begin{rem}
After substitution $a=e^h$ and the Taylor expansion in $h$ the
factor $\frac{a-a^{-1}}{z}$ becomes $ \frac{2h + \dots}{z}$. In
other words its total degree in $h$ and $z$ is not negative.
Therefore, the total degree $k+l$ of the monomial $h^kz^l$ of $P(L)$
is not negative, however the exponent $l$ of $z$ may be as negative
as $-k+1$.
\end{rem}

Our next goal is to describe the Gauss diagram formulas for
$p_{k,l}(L)$. Note that the case $k=0$ corresponds to the
substitution $a=1$ into the HOMFLYPT polynomial, i.e. to the Conway
polynomial. Thus $p_{0,l}(L)$ are  coefficients of the Conway
polynomial for which the Gauss diagram formulas were found in
\cite{CKR}. Thus our work may be considered as a generalization of
\cite{CKR}.

\subsection{Gauss diagram formulas for Vassiliev invariants} \label{s:arrow-diagr}
Let $\A$ be a free $\Z$-module generated by ordered arrow diagrams
with $m$ circles.  Define a map $I:\A\to\A$ by
$I(G):=\sum_{A\subseteq G} A$ for any (abstract, ordered) Gauss
diagram $G$, and extend it to $\A$ by linearity. Here $A\subseteq G$
means the arrow subdiagram $A$ containing the same circles as the
whole diagram $G$ but only a subset of arrows of $G$ with their
signs. A natural scalar product on $\A$ is given by $(A,B):=0$ if
$A$ is not equal to $B$, and $(A,B):=1$ if $A=B$ for a pair of arrow
diagrams $A$ and $B$. Let us define a pairing
$\scp{A}{G}:=(A,I(G))$.

\medskip
\begin{defn}
Let $A$ be a fixed element of $\A$. By {\it a Gauss diagram formula}
we mean a function $\I_A$ on abstract Gauss diagrams defined by
$\I_A: G\mapsto \scp{A}{G}$.
\end{defn}

\medskip
If $A$ is chosen at random then $\I_A(G)$ usually changes under
Reidemeister moves and thus does not define any link invariant.
However, for some special choice of $A$ it might be a link
invariant. According to \cite{GPV} any Vassiliev invariant of long
knots can be expressed by a Gauss diagram formula. In the following
sections we describe an algorithm for finding such formulas for
invariants $p_{k,l}(L)$ coming from the HOMFLYPT polynomial.

For shortness of notation, further we will use {\it unsigned arrow
diagrams}, understanding by that a linear combination of arrow
diagrams with all possible choices of signs and appearing with a
coefficients $\pm1$ depending on whether even or odd number of
negative signs were chosen.

\bigskip
\begin{ex} If $m=2$ and
$$A\ =\ \risS{-12}{adln}{}{70}{15}{12}\ :=\
 \risS{-12}{adln-p}{}{70}{0}{0}\ -\ \risS{-12}{adln-m}{}{70}{0}{0}\ ,
$$
then $\I_A(G)$ is equal to the linking number of components.

If $m=1$ and
$$A\ =\ \risS{-10}{gd2}{}{25}{0}{12}\ :=\ \risS{-10}{gd2-pp}{}{25}{0}{0}\
    -\ \risS{-10}{gd2-pm}{}{25}{0}{0}\ -\ \risS{-10}{gd2-mp}{}{25}{0}{0}\
    +\ \risS{-10}{gd2-mm}{}{25}{0}{0}\ ,
$$
then $\I_A(G)$ is equal to the second coefficient of the Conway
polynomial, $p_{0,2}(G)$ (see \cite{PV}).
\end{ex}

\section{Gauss diagram formulas for HOMFLYPT coefficients} \label{s:result}

Our aim is to figure out contributions of various arrow subdiagrams
to $p_{k,l}$, using the state model from Section \ref{s:jaeger-ga}.

Consider a state model on an arrow diagram $A$ with the following
table of local weights $\lwa$:
\begin{equation}\label{eq:lwa}
\begin{array}{c||c|c|c|c}
\rb{7pt}{First passage:}&
 \risS{0}{fp-b}{}{35}{25}{3}&\risS{0}{fp-t}{}{35}{0}{0}&
 \risS{0}{fp-l}{}{35}{0}{0}&\risS{0}{fp-r}{}{35}{0}{0} \\ \hline\hline
\risS{-8}{cr_p-gd}{}{40}{22}{15} & e^{-h}z & 0 & e^{-2h}-1 & 0 \\
\hline \risS{-8}{cr_m-gd}{}{40}{22}{15} & -e^hz & 0 & e^{2h}-1 & 0
\end{array}
\end{equation}

Let $\plwa=\prod_{\a\in A}\lwa$ and define a power series in $h$ and
$z$ by
\begin{equation}\label{e:coeff}
W(A)=\sum_S \plwa \left(\frac{e^h-e^{-h}}{z}\right)^{c(S)-1}
\end{equation}
Denote by $w_{k,l}$ the coefficient of $h^kz^l$ in $W(A)$, so that
$W(A)=\sum_{k,l} w_{k,l}(A) h^k z^l$.

\medskip
\begin{defn}
Now the linear combination $A_{k,l}\in\A$ can be defined as
follows.\vspace{-5pt}
$$\displaystyle A_{k,l} := \sum\ w_{k,l}(A)\cdot A$$
\end{defn}

\begin{thm}\label{th:main} Let $G$ be a Gauss diagram of an ordered
 link $L$. Then
$$p_{k,l}(L)=\I_{A_{k,l}}(G)=\scp{A_{k,l}}{G}\ .$$
\end{thm}
\begin{proof}
According to \ref{s:jaeger-ga}, the HOMFLYPT is equal to
$$P(G)=\sum_{S\subset G}\ \ \plwg\cdot \left(\frac{a-a^{-1}}{z}\right)^{c(S)-1}.
$$
We have
\begin{multline*}
\plwg=\prod_{\a\in G} \lwg=\prod_{\a\in S} \lwg \prod_{\a\in G\sminus S} \lwg=\\
=\prod_{\a\in S} \lwg \sum_{A\supset S}\left(\prod_{\a\in A\sminus S} (\lwg-1) \prod_{\a\in G\sminus A} 1\right)=\\
=\sum_{A\supset S}\left(\prod_{\a\in S} \lwg\prod_{\a\in A\sminus S}
(\lwg-1)\right).
\end{multline*}
Therefore
\begin{multline*}
P(G)=\sum_{S\subset G}\ \ \sum_{A\supset S}\left(\prod_{\a\in S} \lwg\prod_{\a\in A\sminus S} (\lwg-1)\right)\cdot \left(\frac{a-a^{-1}}{z}\right)^{c(S)-1}=\\
=\sum_{A\subset G}\ \ \sum_{S\subset A}\left(\prod_{\a\in S}
\lwg\prod_{\a\in A\sminus S} (\lwg-1)\right)\cdot
\left(\frac{a-a^{-1}}{z}\right)^{c(S)-1}
\end{multline*}
Comparing tables \eqref{eq:lwg} and \eqref{eq:lwa} of local weights,
we get
$$\prod_{\a\in S} \lwg\prod_{\a\in A\sminus S} (\lwg-1)=\prod_{\a\in A} \lwa=\plwa$$
Thus $$P(G)=\sum_{A\subset G}\ \ \sum_{S\subset A}\plwa\cdot
\left(\frac{a-a^{-1}}{z}\right)^{c(S)-1}$$ And the theorem follows.
\end{proof}

\subsection{Contributions of various diagrams to $A_{k,l}$}\label{sub:contrib}

A state $S$ of an arrow diagram $A$ is called {\it ascending}, if in
the tracing of $A(S)$ we approach a neighborhood of every arrow (not
only the ones in $S$) first at the arrow head. As easy to see from
the weight table, only ascending states contribute to $W(A)$. In
particular, the first end point of an arrow in $A$ (as we move from
the base point along the orientation) must be an arrow head.

Note that since $e^{\pm 2h}-1=\pm 2h + \mbox{(higher degree terms)}$
and $\pm e^{\mp h}z= \pm z + \mbox{(higher degree terms)}$, the
power series $W(A)$ starts with terms of degree at least $|A|$, the
number of arrows of $A$. Moreover, the $z$-power of $\plwa
\left(\frac{e^h-e^{-h}}{z}\right)^{c(S)-1}$ is equal to
$|S|-c(S)+1$. Therefore, for fixed $k$ and $l$, the weight
$w_{k,l}(A)$ of an arrow diagram may be non-zero only if $A$
satisfies the following conditions:
\begin{itemize}
\item[(i)]  $|A|$ is at most $k+l$;
\item[(ii)] there is an ascending state $S$ such that $c(S)=|S|+1-l$.\label{con-iii}
\end{itemize}

For diagrams of the highest degree $|A|=k+l$, the contribution of an
ascending state $S$ to $w_{k,l}(A)$ is equal to
$(-1)^{|A|-|S|}2^k\e(A)$, where $\e(A)$ is the product of signs of
all arrows in $A$. If two such arrow diagrams $A$ and $A'$ with
$|A|=k+l$ differ only by signs of arrows, their contributions to
$A_{k,l}$ differ by the sign $\e(A)\e(A')$. Thus all such diagrams
may be combined to the unsigned diagram $A$, appearing in $A_{k,l}$
with the coefficient $\sum_S(-1)^{|A|-|S|}2^k$ (where the summation
is over all ascending states of $A$ with $c(S)=|S|+1-l$).

Arrow diagrams with isolated arrows do not contribute to $A_{k,l}$.
Indeed, consider an arrow diagram $A\cup a$ with an isolated arrow
$a$. 
Every state $S$ of $A$ corresponds to two states of $A\cup
a$: $S$ and $S\cup a$.
Depending on the orientation of $\a$ their
weights will be either both $0$, or $(e^{-2\e h}-1)\plwa$ and $\e
e^{-\e h}z\frac{e^h-e^{-h}}{z}\plwa$. In both cases they sum up to
$0$, since $(e^{-2\e h}-1)+\e e^{-\e h}(e^h-e^{-h})=0$.

\subsection{Coefficients of the Conway polynomial}
The Conway polynomial is obtained from the HOMFLYPT polynomial by
setting $h=0$. So our formulas for $A_{0,l}$ are the Gauss diagram
formulas for coefficients of the Conway polynomial, discovered
earlier by Michael Khoury and Alfred Rossi \cite{CKR}. Indeed, only
states with $|S|=|A|$ and $c(S)=1$ contribute to $w_{0,l}(A)$. Since
these are diagrams of the highest degree, according to
\ref{sub:contrib} they may be combined into unsigned ascending
diagrams which appear with coefficients 1.

For example, in the case $m=1$ of long knots, states with $c(S)=1$
exist only for even number $l$ of arrows. For $l=2$ and $l=4$ the
resulting linear combinations $A_{0,l}$ are

$$\begin{array}{rcl}
A_{0,2} &=& \risS{-12}{cd22arw}{}{25}{15}{20}\ ;\hspace{1cm}
         A_{0,3} = 0\ ;\\
A_{0,4} &=& \ard{cd4-01arw}\
  +\ \ard{cd4-07arw1} + \ard{cd4-07arw2} + \ard{cd4-07arw3} + \ard{cd4-07arw4}+ \\
&&\hspace{-8pt}
   + \ard{cd4-05arw1} + \ard{cd4-05arw2} + \ard{cd4-05arw3} + \ard{cd4-05arw4}
   + \ard{cd4-05arw5} + \ard{cd4-05arw6} + \ard{cd4-05arw7} + \ard{cd4-05arw8} + \\
&&\hspace{-8pt}
   + \ard{cd4-06arw1} + \ard{cd4-06arw2} + \ard{cd4-06arw3} + \ard{cd4-06arw4}
   + \ard{cd4-06arw5} + \ard{cd4-06arw6} + \ard{cd4-06arw7} + \ard{cd4-06arw8}\ .
\end{array}$$

\section{Low degree examples}\label{s:example}

Let us describe the corresponding formulas for degree 2 and 3
invariants of knots, i.e. $k+l=2,3$, $m=1$. The case $A_{0,2}$ was
described above. A direct check shows that $A_{2,0}=0$. Let us
explicitly find the formula for $A_{1,2}$. The maximal number of
arrows is equal to 3. To get $z^2$ in $W(A)$ we need ascending
states with either $|S|=2$ and $c(S)=1$, or $|S|=3$ and $c(S)=2$. In
the first case the equation $c(S)=1$ means that the two arrows of
$S$ must intersect. In the second case the equation $c(S)=2$ does
not add any restrictions on the relative position of arrows. In
cases $|S|=|A|=2$ or $|S|=|A|=3$, since $S$ is ascending, $A$ itself
must be ascending as well.

For diagrams of the highest degree $|A|=1+2=3$, we should count
ascending states of unsigned arrow diagrams with the coefficient
$(-1)^{3-|S|}2$, i.e. $-2$ for $|S|=2$ and $+2$ for $|S|=3$. There
are only four types of (unsigned) 3-arrow diagrams with no isolated
arrows:
\def\lhd#1{\ \ \risS{-13}{#1}{}{25}{18}{10}\ \ }
\def\fhd#1{\risS{-20}{#1}{}{25}{12}{25}}
\def\shd#1{\ \risS{-12}{#1}{}{25}{15}{15}}
$$\lhd{bcd35-1}; \qquad \risS{-10}{bcd34-3}{}{27.5}{0}{0}\ \ , \qquad \lhd{bcd34-1},\qquad \risS{-10}{bcd34-2}{}{27.5}{0}{0}\ \ .
$$
Diagrams of the same type differ by directions of arrows.

For the first type, recall that the first arrow should be oriented
towards the base point; this leaves 4 possibilities for directions
of the remaining two arrows. One of them, namely $\lhd{aiv-n}$ does
not have ascending states with $|S|=2,3$. The remaining
possibilities, together with their ascending states, are shown in
the table:
$$\begin{array}[t]{||c|c|c|c||}\hline\hline
 \fhd{aiv-3} & \fhd{aiv-2} & \fhd{aiv-1} & \fhd{aiv-1} \\
 \shd{civ-3} & \shd{civ-2} & \shd{civ-1} & \shd{civ-4} \\
\hline\hline\end{array}
$$
The final contribution of this type of 3-arrow diagrams to $A_{1,2}$
is equal to\\
$$-2\ \risS{-12}{aiv-3}{}{25}{0}{15}\
  -2\ \risS{-12}{aiv-2}{}{25}{0}{0}\ .
$$

The remaining three types of 3-arrow diagrams differ by the location
of the base point. A similar consideration shows that 5 out of the
total of 12 arrow diagrams of these types, namely
$$\risS{-10}{aiii-n1}{}{27.5}{15}{15}\ ,\quad
  \risS{-10}{aiii-n2}{}{27.5}{0}{0}\ ,\qquad
  \risS{-13}{ai-n}{}{25}{0}{0}\ ,\qquad
  \risS{-10}{aii-n1}{}{27.5}{0}{0}\ ,\quad
  \risS{-10}{aii-n2}{}{27.5}{0}{0}
$$
do not have ascending states with $|S|=2,3$. The remaining
possibilities, together with their ascending states, are shown in
the table:
$$\begin{array}[t]{||c|c|c||c|c|c||c|c|c||}\hline\hline
 \risS{-17}{aiii-1}{}{27.5}{10}{25}& \risS{-17}{aiii-2}{}{27.5}{0}{0}
      & \risS{-17}{aiii-2}{}{27.5}{0}{0}
& \fhd{ai-1} & \fhd{ai-2} & \fhd{ai-3} &
 \risS{-17}{aii-2}{}{27.5}{0}{0}& \risS{-17}{aii-1}{}{27.5}{0}{0}
      & \risS{-17}{aii-3}{}{27.5}{0}{0} \\
 \risS{-9}{ciii-1}{}{27.5}{15}{15}& \risS{-9}{ciii-2}{}{27.5}{0}{0}
      & \risS{-9}{ciii-3}{}{27.5}{0}{0}
& \shd{ci-1} & \shd{ci-2} & \shd{ci-3} &
 \risS{-9}{cii-2}{}{27.5}{0}{0}& \risS{-9}{cii-1}{}{27.5}{0}{0}
      & \risS{-9}{cii-3}{}{27.5}{0}{0} \\
\hline\hline\end{array}
$$

The final contribution of this type of 3-arrow diagrams to $A_{1,2}$
is equal to
$$-2\ \risS{-10}{aiii-1}{}{27.5}{13}{15}
  -2\ \risS{-13}{ai-1}{}{25}{0}{0}
  -2\ \risS{-13}{ai-2}{}{25}{0}{0}
  +2\ \risS{-13}{ai-3}{}{25}{0}{0}
  -2\ \risS{-10}{aii-2}{}{27.5}{0}{0}\ .
$$

Besides 3-arrow diagrams, some 2-arrow diagrams contribute to
$A_{1,2}$ as well. Since $|A|=2<k+l=3$, contributions of 2-arrow
diagrams depend also on their signs. Such diagrams must be ascending
(since $|S|=|A|=2$) and should not have isolated arrows. There are
four such diagrams, looking like $\lhd{ad2}$\!\!,\vspace{4pt} but
with different signs $\e_1$, $\e_2$ of arrows. For each of them
$\plwa=\e_1\e_2 e^{-(\e_1+\e_2)h}z^2$. If $\e_1=-\e_2$, then
$\plwa=-z^2$, so the coefficient of $hz^2$ vanishes and such
diagrams do not occur in $A_{1,2}$. For two remaining diagrams with
$\e_1=\e_2=\pm1$, coefficients of $hz^2$ in $\plwa$ are equal to
$\mp2$ respectively.

Combining all the above contributions, we finally get
$$A_{1,2} = -2\Bigl( \risS{-12}{aiv-3}{}{25}{0}{15}+
   \risS{-12}{aiv-2}{}{25}{0}{0}+\risS{-10}{aiii-1}{}{27.5}{13}{15}
  +\risS{-13}{ai-1}{}{25}{0}{0}+\risS{-13}{ai-2}{}{25}{0}{0}
  -\risS{-13}{ai-3}{}{25}{0}{0}+\risS{-10}{aii-2}{}{27.5}{0}{0}
    + \shd{ad2pp} - \shd{ad2mm}\Bigr)\ .
$$

The invariant $\I_{A_{1,2}}=\scp{A_{1,2}}{\cdot}$ can be simplified
further. Note that for any classical Gauss diagram $G$, $\scp{
\risS{-13}{ai-2}{}{25}{17}{13}}{G}=
 \scp{ \risS{-13}{ai-3}{}{25}{0}{0} }{G}$.
This follows from the symmetry of the linking number. Indeed,
supposed we have matched two vertical arrows (which are the same in
both diagrams) with two arrows of $G$. Let us consider the
orientation preserving smoothings of the corresponding two crossings
of the link diagram $D$ associated with $G$. The smoothened diagram
$\widetilde{D}$ will have three components. Matchings of the
horizontal arrow of our arrow diagrams with an arrow of $G$ both
measure the linking number between the first and the third
components of $\widetilde{D}$, using crossings when the first
component overpasses (underpasses, respectively) the third
one.\vspace{4pt} Thus, as functions on classical Gauss diagrams,
$\risS{-13}{ai-2}{}{25}{17}{5}$\ is equal to
$\risS{-13}{ai-3}{}{25}{17}{5}$\ and we have
$$p_{1,2}(G) = -2\langle \risS{-12}{aiv-3}{}{25}{0}{15}+
   \risS{-12}{aiv-2}{}{25}{0}{0}+\risS{-10}{aiii-1}{}{27.5}{13}{15}
  +\risS{-13}{ai-1}{}{25}{0}{0}+\risS{-10}{aii-2}{}{27.5}{0}{0}
    + \shd{ad2pp} - \shd{ad2mm}\ , G\rangle\ .
$$

In a similar way one may check that $A_{3,0}=-4 A_{1,2}$.

\begin{ex}
Let us compute the coefficients of $hz^2$ and $h^3$ of the HOMFLYPT
polynomial on the trefoil from Section \ref{d3-1}, see page
\pageref{d3-1}.
$$\scp{A_{1,2}}{G}= 2\scp{\shd{ad2mm}}{G}=2\qquad
\mbox{and}\qquad \scp{A_{3,0}}{G}= -8\ ,
$$
It is easy to verify these coefficients in the Taylor expansion of
$P(3_1)=(2e^{2h}-e^{4h})+ e^{2h}z^2$.
\end{ex}


\begin{thebibliography}{ABC}

\bibitem[BN]{BN} D.~Bar-Natan, {\em On the Vassiliev knot invariants},
     Topology, {\bf 34} (1995) 423--472.

\bibitem[BL]{BL} J.~S.~Birman and X.-S.~Lin, {\it Knot polynomials
       and Vassiliev's invariants}, Invent. Math. {\bf 111} (1993) 225--270.

\bibitem[CDBooK]{CDbook} S.~Chmutov, S.~Duzhin, J.~Mostovoy,
   {\it CDBooK. Introduction to Vassiliev Knot invariants.}
   (a preliminary draft version of a book about Chord Diagrams.)
    \verb#http://www.math.ohio-state.edu/~chmutov/preprints/#.

\bibitem[CKR]{CKR} S.~Chmutov, M.~Khoury, A.~Rossi, {\it Polyak-Viro formulas for
   coefficients of the Conway polynomial}. Preprint {\tt arXiv:math.GT/0810.3146}.
   To appear in the Journal of Knot Theory and its Ramifications.

\bibitem[G1]{G} M.~Goussarov, {\it Finite type invariants are
   presented by Gauss diagram formulas}, Sankt-Petersburg Department of Steklov
   Mathematical Institute preprint (translated from Russian by O. Viro), December 1998.\\
   \verb#http://www.math.toronto.edu/~drorbn/Goussarov/#.

\bibitem[G2]{G2} M.~Goussarov, {\it On $n$-equivalence of knots and invariants of finite degree}, Topology of manifolds and varieties, Adv. Soviet Math. {\bf 18} AMS (1994) 173--192.

\bibitem[GPV]{GPV} M.~Goussarov, M.~Polyak and O.~Viro, {\it Finite type
   invariants of classical and virtual knots}, Topology {\bf 39} (2000) 1045--1068.

\bibitem[HOM]{HOM} P.~Freyd, D.~Yetter, J.~Hoste, W.~B.~R.~Lickorish,
   K.~Millett, A.~Ocneanu, {\it A new polynomial invariant of knots and links},
   Bull. AMS {\bf 12} (1985) 239--246.

\bibitem[H]{H} J.~Hoste, {\it A polynomial invariant of knots and links},
   Pacific J. Math.  {\bf 124}(2) (1986) 295--320.

\bibitem[J]{Ja} F.~Jaeger, {\it A combinatorial model for the Homfly polynomial},
   European J.Combinatorics {\bf 11} (1990) 549--558.

\bibitem[LM]{LM} W.~B.~R.~Lickorish, K.~Millett, {\it A polynomial invariant of
   oriented links}, Topology {\bf 26}(1) (1987) 107--141.

\bibitem[\"{O}]{Oll} Olof-Petter \"{O}stlund, {\it
   Invariants of knot diagrams and relations among Reidemeister moves},
   Journal of Knot Theory and its Ramifications, {\bf 10}(8) (2001)
   1215--1227. Preprint {\tt arXiv:math.GT/0005108}.

\bibitem[PV]{PV} M.~Polyak and O.~Viro,
   {\it Gauss diagram formulas for Vassiliev invariants},
   Int. Math. Res. Notes {\bf 11} (1994) 445--454.

\bibitem[PT]{PT} J.~Przytycki, P.~Traczyk, {\it Invariants of links
   of the Conway type}, Kobe J. Math. {\bf 4} (1988) 115--139.

\end{thebibliography}
\end{document}